\documentclass[10pt,twoside]{article}
\usepackage{amsmath,amsthm,latexsym,amssymb,hyperref}
\usepackage{datetime}
\usepackage{dsfont}
\usepackage{mathrsfs}
\usepackage{enumitem}
\usepackage{listings}
\usepackage[usenames,dvipsnames]{color} 
\usepackage{color,graphicx}
\newcommand{\quotes}[1]{``#1''}

\theoremstyle{plain} 
\newtheorem{theorem}{Theorem}[section]

\newtheorem{proposition}[theorem]{Proposition}

\theoremstyle{definition} 
\newtheorem{definition}[theorem]{Definition}

\newtheorem{example}[theorem]{Example}

\begin{document}

\title{\bf Introducing fully UP-semigroups\footnote{This work was financially supported by the University of Phayao.}}

\author{Aiyared Iampan\footnote{Corresponding author. Email: \texttt{aiyared.ia@up.ac.th}} \\[.3cm] Department of Mathematics, School of Science \\ University of Phayao, Phayao 56000, Thailand}\maketitle

\noindent\hrulefill

\begin{abstract}
In this paper, we introduce some new classes of algebras related to UP-algebras and semigroups, called a left UP-semigroup, a right UP-semigroup, a fully UP-semigroup, a left-left UP-semigroup, a right-left UP-semigroup, a left-right UP-semigroup, a right-right UP-semigroup, a fully-left UP-semigroup, a fully-right UP-semigroup, a left-fully UP-semigroup, a right-fully UP-semigroup, a fully-fully UP-semigroup, and find their examples.
\end{abstract}

\begin{flushleft}
\textbf{Mathematics Subject Classification:} 08A99, 03G25 \\
\textbf{Keywords:} semigroup, UP-algebra, fully UP-semigroup
\end{flushleft}

\noindent\hrulefill


\section{Introduction and Preliminaries}\numberwithin{equation}{section}


In the literature, several researches introduced a new class of algebras related to logical algebras and semigroups such as:
In 1993, Jun, Hong and Roh \cite{Jun1993pp59} introduced the notion of BCI-semigroups.
In 1998, Jun, Xin and Roh \cite{Jun1998pp309} renamed the BCI-semigroup as the IS-algebra.
In 2006, Kim \cite{Kim2006pp67} introduced the notion of KS-semigroups.
In 2011, Ahn and Kim \cite{Ahn2011pp8} introduced the notion of BE-semigroups.
In 2015, Endam and Vilela \cite{Endam2015pp2901} introduced the notion of JB-semigroups.
In 2016, Sultana and Chaudhary \cite{Sultana2016pp1} introduced the notion of BCH-semigroups.
In this paper, we introduce some new classes of algebras related to UP-algebras and semigroups, called a left UP-semigroup, a right UP-semigroup, a fully UP-semigroup, a left-left UP-semigroup, a right-left UP-semigroup, a left-right UP-semigroup, a right-right UP-semigroup, a fully-left UP-semigroup, a fully-right UP-semigroup, a left-fully UP-semigroup, a right-fully UP-semigroup, a fully-fully UP-semigroup, and find their examples.

\medskip

Before we begin our study, we will introduce the definition of a UP-algebra.


\begin{definition}\label{UP1}\cite{Iampan2014}
An algebra $A=(A,\cdot,0)$ of type $(2,0)$ is called a \textit{UP-algebra}, where $A$ is a nonempty set, $\cdot$ is a binary operation on $A$, and $0$ is a fixed element of $A$ (i.e., a nullary operation) if it satisfies the following axioms: for any $x,y,z\in A$,
\begin{description}
\item[(UP-1)] $(y\cdot z)\cdot((x\cdot y)\cdot(x\cdot z))=0$,
\item[(UP-2)] $0\cdot x=x$,
\item[(UP-3)] $x\cdot0=0$, and
\item[(UP-4)] $x\cdot y=y\cdot x=0$ implies $x=y$.
\end{description}
\end{definition}


In a UP-algebra $A=(A,\cdot,0)$, the following assertions are valid (see \cite{Iampan2014,Iampan2018book}).
\begin{align}
&\label{ppt01} (\forall x\in A)(x\cdot x=0), \\
&\label{ppt02} (\forall x,y,z\in A)(x\cdot y=0, y\cdot z=0 \Rightarrow x\cdot z=0), \\
&\label{ppt03} (\forall x,y,z\in A)(x\cdot y=0 \Rightarrow (z\cdot x)\cdot (z\cdot y)=0), \\
&\label{ppt04} (\forall x,y,z\in A)(x\cdot y=0 \Rightarrow (y\cdot z)\cdot (x\cdot z)=0), \\
&\label{ppt05} (\forall x,y\in A)(x\cdot (y\cdot x)=0), \\
&\label{ppt06} (\forall x,y\in A)((y\cdot x)\cdot x=0 \Leftrightarrow x=y\cdot x), \\
&\label{ppt07} (\forall x,y\in A)(x\cdot (y\cdot y)=0), \\
&\label{UP2_add1} (\forall a,x,y,z\in A)((x\cdot(y\cdot z))\cdot(x\cdot((a\cdot y)\cdot (a\cdot z)))=0), \\
&\label{UP2_add2} (\forall a,x,y,z\in A)((((a\cdot x)\cdot (a\cdot y))\cdot z)\cdot((x\cdot y)\cdot z)=0), \\
&\label{UP2_add3} (\forall x,y,z\in A)(((x\cdot y)\cdot z)\cdot(y\cdot z)=0), \\
&\label{UP2_add4} (\forall x,y,z\in A)(x\cdot y=0 \Rightarrow x\cdot (z\cdot y)=0), \\
&\label{UP2_add5} (\forall x,y,z\in A)(((x\cdot y)\cdot z)\cdot(x\cdot(y\cdot z))=0),\, \text{and} \\
&\label{UP2_add6} (\forall a,x,y,z\in A)(((x\cdot y)\cdot z)\cdot(y\cdot(a\cdot z))=0).
 \end{align}


Let $X$ be a universal set. Define two binary operations $\cdot$ and $\ast$ on the power set of $X$ by putting, for all $A,B\in \mathcal{P}(X)$,
\begin{eqnarray}
\label{eq001} A\cdot B &=& A^{\prime}\cap B, \\
\label{eq002} A\ast B &=& A^{\prime}\cup B.
\end{eqnarray}
Then $(\mathcal{P}(X),\cdot,\emptyset)$ is a UP-algebra and we shall call it the \textit{power UP-algebra of type 1} \cite{Iampan2014}, and $(\mathcal{P}(X),\ast,X)$ is a UP-algebra and we shall call it the \textit{power UP-algebra of type 2} \cite{Iampan2014}.
Now, define four binary operations $\odot,\otimes,\boxdot$ and $\boxtimes$ on the power set of $X$ by putting, for all $A,B\in \mathcal{P}(X)$,
\begin{eqnarray}
\label{eq003} A\odot B &=& X, \\
\label{eq004} A\otimes B &=& \emptyset, \\
\label{eq005} A\boxdot B &=& B, \\
\label{eq006} A\boxtimes B &=& A.
\end{eqnarray}
Then $(\mathcal{P}(X),\odot),(\mathcal{P}(X),\otimes),(\mathcal{P}(X),\boxdot)$ and $(\mathcal{P}(X),\boxtimes)$ are semigroups.
Furthermore, we know that $(\mathcal{P}(X),\cap,X)$ and $(\mathcal{P}(X),\cup,\emptyset)$ are monoids.


\begin{definition}\label{Fu001}
Let $A$ be a nonempty set, $\cdot$ and $\ast$ are binary operations on $A$, and $0$ is a fixed element of $A$ (i.e., a nullary operation). An algebra $A=(A,\cdot,\ast,0)$ of type $(2,2,0)$ in which $(A,\cdot,0)$ is a UP-algebra and $(A,\ast)$ is a semigroup is called
\begin{enumerate}[label=\textrm{(\arabic*)}]
\item\label{Fu001_1} a \textit{left UP-semigroup} (in short, an \textit{$l$-UP-semigroup}) if the operation \quotes{$\ast$} is left distributive over the operation \quotes{$\cdot$},
\item\label{Fu001_2} a \textit{right UP-semigroup} (in short, an \textit{$r$-UP-semigroup}) if the operation \quotes{$\ast$} is right distributive over the operation \quotes{$\cdot$},
\item\label{Fu001_3} a \textit{fully UP-semigroup} (in short, an \textit{$f$-UP-semigroup}) if the operation \quotes{$\ast$} is distributive (on both sides) over the operation \quotes{$\cdot$},
\item\label{Fu001_4} a \textit{left-left UP-semigroup} (in short, an \textit{$(l,l)$-UP-semigroup}) if the operation \quotes{$\cdot$} is left distributive over the operation \quotes{$\ast$} and the operation \quotes{$\ast$} is left distributive over the operation \quotes{$\cdot$},
\item\label{Fu001_5} a \textit{right-left UP-semigroup} (in short, an \textit{$(r,l)$-UP-semigroup}) if the operation \quotes{$\cdot$} is right distributive over the operation \quotes{$\ast$} and the operation \quotes{$\ast$} is left distributive over the operation \quotes{$\cdot$},
\item\label{Fu001_6} a \textit{left-right UP-semigroup} (in short, an \textit{$(l,r)$-UP-semigroup}) if the operation \quotes{$\cdot$} is left distributive over the operation \quotes{$\ast$} and the operation \quotes{$\ast$} is right distributive over the operation \quotes{$\cdot$},
\item\label{Fu001_7} a \textit{right-right UP-semigroup} (in short, an \textit{$(r,r)$-UP-semigroup}) if the operation \quotes{$\cdot$} is right distributive over the operation \quotes{$\ast$} and the operation \quotes{$\ast$} is right distributive over the operation \quotes{$\cdot$},
\item\label{Fu001_8} a \textit{fully-left UP-semigroup} (in short, an \textit{$(f,l)$-UP-semigroup}) if the operation \quotes{$\cdot$} is distributive (on both sides) over the operation \quotes{$\ast$} and the operation \quotes{$\ast$} is left distributive over the operation \quotes{$\cdot$},
\item\label{Fu001_9} a \textit{fully-right UP-semigroup} (in short, an \textit{$(f,r)$-UP-semigroup}) if the operation \quotes{$\cdot$} is distributive (on both sides) over the operation \quotes{$\ast$} and the operation \quotes{$\ast$} is right distributive over the operation \quotes{$\cdot$},
\item\label{Fu001_10} a \textit{left-fully UP-semigroup} (in short, an \textit{$(l,f)$-UP-semigroup}) if the operation \quotes{$\cdot$} is left distributive over the operation \quotes{$\ast$} and the operation \quotes{$\ast$} is distributive (on both sides) over the operation \quotes{$\cdot$},
\item\label{Fu001_11} a \textit{right-fully UP-semigroup} (in short, an \textit{$(r,f)$-UP-semigroup}) if the operation \quotes{$\cdot$} is right distributive over the operation \quotes{$\ast$} and the operation \quotes{$\ast$} is distributive (on both sides) over the operation \quotes{$\cdot$}, and
\item\label{Fu001_12} a \textit{fully-fully UP-semigroup} (in short, an \textit{$(f,f)$-UP-semigroup}) if the operation \quotes{$\cdot$} is distributive (on both sides) over the operation \quotes{$\ast$} and the operation \quotes{$\ast$} is distributive (on both sides) over the operation \quotes{$\cdot$}.
\end{enumerate}
\end{definition}



In what follows, let $A$ and $B$ denote UP-algebras unless otherwise specified.
The following proposition is very important for the study of UP-algebras.


\medskip

The proof of Propositions \ref{Fu002}, \ref{Fu003}, \ref{Fu004}, \ref{Fu005}, \ref{Fu006}, and \ref{Fu007} can be verified by a routine proof.

\begin{proposition}\label{Fu002} (The operations of a UP-algebra $\mathcal{P}(X)$ is left distributive over the operations of a semigroup $\mathcal{P}(X)$)
Let $X$ be a universal set. Then the following properties hold: for any $A,B,C\in \mathcal{P}(X)$,
\begin{enumerate}[label=\textrm{(\arabic*)}]
\item\label{Fu002_1} $A\cdot(B\cap C)=(A\cdot B)\cap(A\cdot C)$,
\item\label{Fu002_2} $A\cdot(B\cup C)=(A\cdot B)\cup(A\cdot C)$,
\item\label{Fu002_3} $A\ast(B\cap C)=(A\ast B)\cap(A\ast C)$,
\item\label{Fu002_4} $A\ast(B\cup C)=(A\ast B)\cup(A\ast C)$,
\item\label{Fu002_5} $A\cdot(B\otimes C)=(A\cdot B)\otimes(A\cdot C)$,
\item\label{Fu002_6} $A\ast(B\odot C)=(A\ast B)\odot(A\ast C)$,
\item\label{Fu002_7} $A\cdot(B\boxdot C)=(A\cdot B)\boxdot(A\cdot C)$,
\item\label{Fu002_8} $A\ast(B\boxdot C)=(A\ast B)\boxdot(A\ast C)$,
\item\label{Fu002_9} $A\cdot(B\boxtimes C)=(A\cdot B)\boxtimes(A\cdot C)$, and
\item\label{Fu002_10} $A\ast(B\boxtimes C)=(A\ast B)\boxtimes(A\ast C)$.
\end{enumerate}
\end{proposition}


\begin{proposition}\label{Fu003} (The operations of a UP-algebra $\mathcal{P}(X)$ is right distributive over the operations of a semigroup $\mathcal{P}(X)$)
Let $X$ be a universal set. Then the following properties hold: for any $A,B,C\in \mathcal{P}(X)$,
\begin{enumerate}[label=\textrm{(\arabic*)}]
\item\label{Fu003_1} $(A\boxdot B)\cdot C=(A\cdot C)\boxdot(B\cdot C)$,
\item\label{Fu003_2} $(A\boxdot B)\ast C=(A\ast C)\boxdot(B\ast C)$,
\item\label{Fu003_3} $(A\boxtimes B)\cdot C=(A\cdot C)\boxtimes(B\cdot C)$, and
\item\label{Fu003_4} $(A\boxtimes B)\ast C=(A\ast C)\boxtimes(B\ast C)$.
\end{enumerate}
\end{proposition}


\begin{proposition}\label{Fu004} (The operations of a semigroup $\mathcal{P}(X)$ is left distributive over the operations of a UP-algebra $\mathcal{P}(X)$)
Let $X$ be a universal set. Then the following properties hold: for any $A,B,C\in \mathcal{P}(X)$,
\begin{enumerate}[label=\textrm{(\arabic*)}]
\item\label{Fu004_1} $A\odot(B\ast C)=(A\odot B)\ast(A\odot C)$,
\item\label{Fu004_2} $A\otimes(B\cdot C)=(A\otimes B)\cdot(A\otimes C)$,
\item\label{Fu004_3} $A\boxdot(B\cdot C)=(A\boxdot B)\cdot(A\boxdot C)$, and
\item\label{Fu004_4} $A\boxdot(B\ast C)=(A\boxdot B)\ast(A\boxdot C)$.
\end{enumerate}
\end{proposition}


\begin{proposition}\label{Fu005} (The operations of a semigroup $\mathcal{P}(X)$ is right distributive over the operations of a UP-algebra $\mathcal{P}(X)$)
Let $X$ be a universal set. Then the following properties hold: for any $A,B,C\in \mathcal{P}(X)$,
\begin{enumerate}[label=\textrm{(\arabic*)}]
\item\label{Fu005_1} $(A\ast B)\odot C=(A\odot C)\ast(B\odot C)$,
\item\label{Fu005_2} $(A\cdot B)\otimes C=(A\otimes C)\cdot(B\otimes C)$,
\item\label{Fu005_3} $(A\cdot B)\boxtimes C=(A\boxtimes C)\cdot(B\boxtimes C)$, and
\item\label{Fu005_4} $(A\ast B)\boxtimes C=(A\boxtimes C)\ast(B\boxtimes C)$.
\end{enumerate}
\end{proposition}


\begin{proposition}\label{Fu006} Let $X$ be a universal set. Then the following properties hold: for any $A,B,C\in \mathcal{P}(X)$,
\begin{enumerate}[label=\textrm{(\arabic*)}]
\item\label{Fu006_1} $(A\cap B)\cdot C=(A\cdot C)\cup(B\cdot C)$,
\item\label{Fu006_2} $(A\cup B)\cdot C=(A\cdot C)\cap(B\cdot C)$,
\item\label{Fu006_3} $(A\cap B)\ast C=(A\ast C)\cup(B\ast C)$,
\item\label{Fu006_4} $(A\cup B)\ast C=(A\ast C)\cap(B\ast C)$,
\item\label{Fu006_5} $(A\odot B)\cdot C=(A\cdot C)\otimes(B\cdot C)$, and
\item\label{Fu006_6} $(A\otimes B)\ast C=(A\ast C)\odot(B\ast C)$.
\end{enumerate}
\end{proposition}


\begin{proposition}\label{Fu007} Let $X$ be a universal set. Then the following properties hold: for any $A,B,C\in \mathcal{P}(X)$,
\begin{enumerate}[label=\textrm{(\arabic*)}]
\item\label{Fu007_1} $(A\cdot B)\odot C=(A\otimes C)\ast(B\otimes C)$, and
\item\label{Fu007_2} $(A\ast B)\otimes C=(A\odot C)\cdot(B\odot C)$.
\end{enumerate}
\end{proposition}


\begin{proposition}\label{Fu009} Let $A=(A,\cdot,\ast,0)$ be an algebra of type $(2,2,0)$ in which $(A,\cdot,0)$ is a UP-algebra and $(A,\ast)$ is a semigroup. Then the following properties hold:
\begin{enumerate}[label=\textrm{(\arabic*)}]
\item\label{Fu009_1} if $A$ is an $l$-UP-semigroup, then $x\ast 0=0$ for all $x\in A$,
\item\label{Fu009_2} if $A$ is an $r$-UP-semigroup, then $0\ast x=0$ for all $x\in A$,
\item\label{Fu009_3} if the operation \quotes{$\cdot$} is right distributive over the operation \quotes{$\ast$}, then $x\ast x=x$ for all $x\in A$, and
\item\label{Fu009_4} $A=\{0\}$ is one and only one $(r,f)$-UP-semigroup and $(f,f)$-UP-semigroup.
\end{enumerate}
\end{proposition}

\begin{proof}
\ref*{Fu009_1} Assume that $A$ is an $l$-UP-semigroup. Then, by (\ref{ppt01}), we have
\begin{center}
$x\ast0=x\ast(0\cdot0)=(x\ast0)\cdot(x\ast0)=0$ for all $x\in A$.
\end{center}
\ref*{Fu009_2} Assume that $A$ is an $r$-UP-semigroup. Then, by (\ref{ppt01}), we have
\begin{center}
$0\ast x=(0\cdot0)\ast x=(0\ast x)\cdot(0\ast x)=0$ for all $x\in A$.
\end{center}
\ref*{Fu009_3} Assume that the operation \quotes{$\cdot$} is right distributive over the operation \quotes{$\ast$}.
Then, by (UP-3), we have
\begin{center}
$0=(0\ast0)\cdot0=(0\cdot0)\ast(0\cdot0)=0\ast0$.
\end{center}
Thus, by (UP-2), we have
\begin{center}
$x=0\cdot x=(0\ast0)\cdot x=(0\cdot x)\ast(0\cdot x)=x\ast x$ for all $x\in A$.
\end{center}
\ref*{Fu009_4} By (UP-2), (\ref{ppt01}), \ref{Fu009_1} and \ref{Fu009_2}, we have
\begin{center}
$x=0\cdot x=(x\ast0)\cdot x=(x\cdot x)\ast(0\cdot x)=0\ast x=0$ for all $x\in A$.
\end{center}
Hence, $A=\{0\}$ is one and only one $(r,f)$-UP-semigroup and $(f,f)$-UP-semigroup.
\end{proof}


\begin{example}\label{Fu008} Let $A=\{0,1,2,3\}$ be a set with a binary operation $\cdot$ defined by the following Cayley table:
\begin{equation*}
\begin{array}{c|cccc}
  \cdot & 0 & 1 & 2 & 3 \\
  \hline
  0 & 0 & 1 & 2 & 3 \\
  1 & 0 & 0 & 2 & 3 \\
  2 & 0 & 1 & 0 & 3 \\
  3 & 0 & 1 & 2 & 0
\end{array}
~\text{and}~
\begin{array}{c|cccc}
  \ast & 0 & 1 & 2 & 3 \\
  \hline
  0 & 0 & 0 & 0 & 0 \\
  1 & 0 & 0 & 0 & 0 \\
  2 & 0 & 0 & 0 & 1 \\
  3 & 0 & 0 & 1 & 0
\end{array}
\end{equation*}
Then $(A,\cdot,\ast,0)$ is an $f$-UP-semigroup.
\end{example}

 \clearpage

Let $X$ be a universal set.
Then, by above propositions and an example, we get:

$$\begin{tabular}{|l|l|}
  \hline
\multicolumn{1}{|c|}{\textbf{Types of algebras}} &  \multicolumn{1}{|c|}{\textbf{Examples}} \\
  \hline
$l$-UP-semigroup &  $(\mathcal{P}(X),\ast,\odot,X)$ (see Proposition \ref{Fu004} \ref{Fu004_1}) \\
 &  $(\mathcal{P}(X),\cdot,\otimes,\emptyset)$ (see Proposition \ref{Fu004} \ref{Fu004_2}) \\
 &  $(\mathcal{P}(X),\cdot,\boxdot,\emptyset)$ (see Proposition \ref{Fu004} \ref{Fu004_3}) \\
 &  $(\mathcal{P}(X),\ast,\boxdot,X)$ (see Proposition \ref{Fu004} \ref{Fu004_4}) \\
$r$-UP-semigroup &  $(\mathcal{P}(X),\ast,\odot,X)$ (see Proposition \ref{Fu005} \ref{Fu005_1}) \\
 &  $(\mathcal{P}(X),\cdot,\otimes,\emptyset)$ (see Proposition \ref{Fu005} \ref{Fu005_2}) \\
  &  $(\mathcal{P}(X),\cdot,\boxtimes,\emptyset)$ (see Proposition \ref{Fu005} \ref{Fu005_3}) \\
   &  $(\mathcal{P}(X),\ast,\boxtimes,X)$ (see Proposition \ref{Fu005} \ref{Fu005_4}) \\
$f$-UP-semigroup  & $(\mathcal{P}(X),\ast,\odot,X)$ (see Propositions \ref{Fu004} \ref{Fu004_1} and \ref{Fu005} \ref{Fu005_1}) \\
 &  $(\mathcal{P}(X),\cdot,\otimes,\emptyset)$ (see Propositions \ref{Fu004} \ref{Fu004_2} and \ref{Fu005} \ref{Fu005_2}) \\
  &  $(A,\cdot,\ast,0)$ (see Example \ref{Fu008}) \\
$(l,l)$-UP-semigroup  & $(\mathcal{P}(X),\cdot,\boxdot,\emptyset)$ (see Propositions \ref{Fu004} \ref{Fu004_3} and \ref{Fu002} \ref{Fu002_7}) \\
  & $(\mathcal{P}(X),\ast,\boxdot,X)$ (see Propositions \ref{Fu004} \ref{Fu004_4} and \ref{Fu002} \ref{Fu002_8}) \\
$(r,l)$-UP-semigroup  & $(\mathcal{P}(X),\cdot,\boxdot,\emptyset)$ (see Propositions \ref{Fu004} \ref{Fu004_3} and \ref{Fu003} \ref{Fu003_1}) \\
  & $(\mathcal{P}(X),\ast,\boxdot,X)$ (see Propositions \ref{Fu004} \ref{Fu004_4} and \ref{Fu003} \ref{Fu003_2}) \\
$(l,r)$-UP-semigroup  &  $(\mathcal{P}(X),\ast,\odot,X)$ (see Propositions \ref{Fu005} \ref{Fu005_1} and \ref{Fu002} \ref{Fu002_6}) \\
 &  $(\mathcal{P}(X),\cdot,\otimes,\emptyset)$ (see Propositions \ref{Fu005} \ref{Fu005_2} and \ref{Fu002} \ref{Fu002_5}) \\
  &  $(\mathcal{P}(X),\cdot,\boxtimes,\emptyset)$ (see Propositions \ref{Fu005} \ref{Fu005_3} and \ref{Fu002} \ref{Fu002_9}) \\
   &  $(\mathcal{P}(X),\ast,\boxtimes,X)$ (see Propositions \ref{Fu005} \ref{Fu005_4} and \ref{Fu002} \ref{Fu002_10}) \\
$(r,r)$-UP-semigroup  &  $(\mathcal{P}(X),\cdot,\boxtimes,\emptyset)$ (see Propositions \ref{Fu005} \ref{Fu005_3} and \ref{Fu003} \ref{Fu003_3}) \\
   &  $(\mathcal{P}(X),\ast,\boxtimes,X)$ (see Propositions \ref{Fu005} \ref{Fu005_4} and \ref{Fu003} \ref{Fu003_4}) \\
$(f,l)$-UP-semigroup  &  $(\mathcal{P}(X),\cdot,\boxdot,\emptyset)$ (see Propositions \ref{Fu004} \ref{Fu004_3}, \ref{Fu002} \ref{Fu002_7}, and \ref{Fu003} \ref{Fu003_1}) \\
  & $(\mathcal{P}(X),\ast,\boxdot,X)$ (see Propositions \ref{Fu004} \ref{Fu004_4}, \ref{Fu002} \ref{Fu002_8}, and \ref{Fu003} \ref{Fu003_2}) \\
$(f,r)$-UP-semigroup  &  $(\mathcal{P}(X),\cdot,\boxtimes,\emptyset)$ (see Propositions \ref{Fu005} \ref{Fu005_3}, \ref{Fu002} \ref{Fu002_9}, and \ref{Fu003} \ref{Fu003_3}) \\
   &  $(\mathcal{P}(X),\ast,\boxtimes,X)$ (see Propositions \ref{Fu005} \ref{Fu005_4}, \ref{Fu002} \ref{Fu002_10}, and \ref{Fu003} \ref{Fu003_4}) \\
$(l,f)$-UP-semigroup  &  $(\mathcal{P}(X),\ast,\odot,X)$ (see Propositions \ref{Fu004} \ref{Fu004_1}, \ref{Fu002} \ref{Fu002_6}, and \ref{Fu005} \ref{Fu005_1}) \\
 &  $(\mathcal{P}(X),\cdot,\otimes,\emptyset)$ (see Propositions \ref{Fu004} \ref{Fu004_2}, \ref{Fu002} \ref{Fu002_5}, and \ref{Fu005} \ref{Fu005_2}) \\
$(r,f)$-UP-semigroup  &  $\{0\}$ is one and only one $(r,f)$-UP-semigroup  \\
$(f,f)$-UP-semigroup  &  $\{0\}$ is one and only one $(f,f)$-UP-semigroup \\
    \hline
 \end{tabular}$$

 \clearpage


Hence, we have the following diagram:

\begin{figure}[h]
\centering
\includegraphics[width=0.9\textwidth]{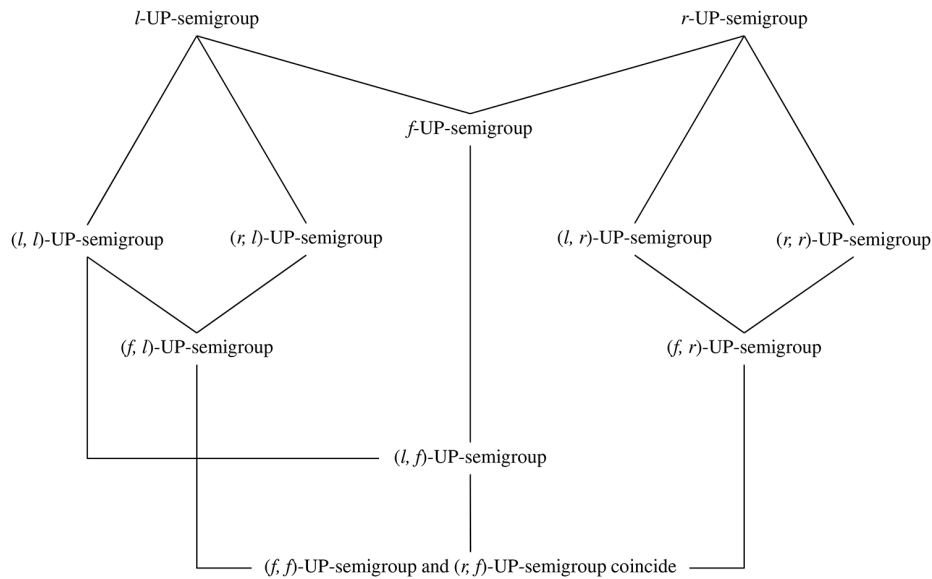}
\caption{New algebras of type (2,2,0)}
\end{figure}


\section*{Conclusion}

We have introduced the notions of left UP-semigroups, right UP-semigroups, fully UP-semigroups, left-left UP-semigroups, right-left UP-semigroups, left-right UP-semigroups, right-right UP-semigroups, fully-left UP-semigroups, fully-right UP-semigroups, left-fully UP-semigroups, right-fully UP-semigroups and fully-fully UP-semigroups, and have found examples.
We have that right-fully UP-semigroups and fully-fully UP-semigroups coincide, and it is only $\{0\}$.
In further study, we will apply the notion of fuzzy sets and fuzzy soft sets to the theory of all above notions.


\section*{Acknowledgment}
The author wish to express their sincere thanks to the referees for the valuable suggestions which lead to an improvement of this paper.


{\bf Received: \today}

\end{document}